\documentclass[12pt]{article}
\usepackage{amssymb}
\usepackage{amsmath}
\usepackage{amsthm}
\DeclareSymbolFont{extraup}{U}{zavm}{m}{n}
\DeclareMathSymbol{\clubsuit}{\mathalpha}{extraup}{88}

\numberwithin{equation}{section}

\newtheorem{thm}{Theorem}[section] 
\newtheorem{lem}[thm]{Lemma}

\newtheorem*{conj}{Conjecture}
\newtheorem{prob}{Problem}
\newtheorem*{claim}{Claim}

\title{Towards a Baranyai theorem \\ with additional condition}
\author{{\bf Gyula O.H. Katona}\thanks{The work of this author was
supported by the Hungarian National Research, Development and Innovation Office
NKFIH under grant numbers SSN135643 and K132696.}\\
R\'enyi Institute, HUN-REN\\
Budapest Pf 127, 1364 Hungary \\ ohkatona@renyi.hu  \and {\bf
Gyula Y. Katona\thanks{The research of this author  is supported by the BME-Artificial Intelligence FIKP grant of EMMI (BME FIKP-MI/SC), and the research  project K-132696 by the Ministry of Innovation and Technology of Hungary from the National Research, Development and Innovation Fund.}} \\
Department of Computer Science \\
Budapest University of Technology and Economics\\
Budapest, Magyar tud\'osok krt. 2., 1117 Hungary \\
katona.gyula@vik.bme.hu}

\begin{document}
\date{}

\maketitle

\begin{abstract}
A $(k,\ell )$ partial partition of an $n$-element set is a collection of $\ell $ 
pairwise disjoint $k$-element subsets. It is proved that, if $n$ is large enough, one can find 
$\left\lfloor 
{n\choose k}/{\ell}\right\rfloor$ 
such partial partitions in such a way
that if $A_1$ and $A_2$ are distinct  classes in one of the partial partitions,
$B_1$ and $B_2$ are distinct classes in another one, then one of the intersections
$A_1\cap B_1, A_2\cap B_2$ has size at most ${k\over 2}$.


\end{abstract}

\section{Introduction}
Let $[n]=\{ 1,2,\ldots n\}$ be an $n$-element set. Let us start with the
classic theorem of Baranyai.

\begin{thm} {\rm (\cite{Ba})} If $k$ divides $m$ then there is a set of
partitions of $[m]$ into $k$-element classes such that each element of
${[m]\choose k}$ is contained in exactly one such partition.
\end{thm}

But what do we know about the relationship between two such partitions?
Can we impose some restrictions of this kind? This question seems
to be very difficult but becomes attackable if  partitions are
replaced by the following weaker concept. A $(k,\ell)$ partial partition
or $(k,\ell)$-{\it papartition} is a family of $\ell$ pairwise disjoint
$k$-element sets.  Our restriction will be the following. We say that two
$(k,\ell)$-papartitions are {\it too close} if they contain two distinct members
$A_1, B_1$ in the first papartition and two distinct members
$A_2, B_2$ in the second papartition satisfying
$$|A_1\cap A_2|>{k\over 2},\quad |B_1\cap B_2|>{k\over 2}.$$
Our condition will be that no two papartitions are too close.

In Section 2, we will state our main theorem and its generalization for graphs.
Section 3 contains the proofs, while Section 4 adds some remarks on related problems.

\section{Results}

\begin{thm}\label{thm1}
Let $k$ and $\ell$ be positive integers. If $n$ is large enough then
one can find
$$\left\lfloor \frac{{n\choose k}}{\ell}\right\rfloor$$
$(k,\ell)$-papartitions in such a way that no $k$-element set appears in two
of them and they are not too close to each other.
\end{thm}

Now let us see the general theorem for graphs. We  have two
graphs on the same vertex set: $G_1=(V,E_1), G_2=(V,E_2)$.
The edges of $G_1$ and $G_2$ are called {\it blue} and {\it red},
respectively. Denote the minimum degree in $G_1$ by $\delta_1$,
the maximum degree in $G_2$ by $\Delta_2$.
An \emph{alternating-$(\ell,2,\ell,2)$-bag} consists of two
vertex-disjoint copies of $K_{\ell}$ in $G_1$ (i.e. all blue edges) and two vertex-disjoint
edges in $G_2$ (i.e. two red edges) both connecting vertices in different copies of $K_{\ell}$.
\begin{thm}\label{thm2} Let $G_1=(V,E_1)$ and $G_2=(V,E_2)$ be two graphs
where $|V|=m$ and $E_1\cap E_2=\emptyset,\ \ell \geq 2$ is a given integer.
Let the minimum degree in $G_1$ be
$\delta_1$, and the maximum degree in $G_2$ be $\Delta_2$. If
\begin{equation}\label{eq21}
m\left({\ell^2-1\over \ell^2}+\alpha \right)\leq \delta_1\quad {\rm{where}}\quad \ 0<\alpha <{1\over \ell^2}
\end{equation}
and
\begin{equation}\label{eq22}
	\Delta_2\leq \sqrt{{m\alpha -\ell \over 3}}
\end{equation}
hold, $m\geq m(\ell)$ is large enough, then there are $\lfloor {m/\ell}\rfloor$
vertex-disjoint copies of $K_{\ell}$ in $G_1$ so that no
two of these copies span an alternating-$(\ell,2,\ell,2)$-bag.
\end{thm}

\section{Proofs}
First we prove Theorem \ref{thm2} by a series of Lemmas.
\begin{lem}\label{lem1} Let $G=(V,E)$ be a graph with $|V|=m$. Suppose that
$0< \alpha < {1/ {\ell}^2}$ and the minimum degree
in $G$ is at least
\begin{equation}\label{eq31}
m\left(\frac{\ell^2-1}{\ell^2}+\alpha \right) .
\end{equation}
If  $m$ is large enough then $G$ contains $\lfloor {m/\ell}\rfloor$ vertex-disjoint copies
of $K_{\ell}$.
\end{lem}
\begin{proof} We will call such a system of vertex-disjoint copies of $K_{\ell}$ an 
\emph{almost-$\ell $-decomposition} of the graph $G$. If $G=K_m$ then the statement is obviously true. Suppose that $G$ is
a counter-example and add edges one by one until the statement is true.
Let $G^{\prime}$ be this graph and $e$ be the last edge added. Then
$G^{\prime}-e$ is also a counter-example, $G^{\prime}$ contains an
almost-$\ell$-decomposition,
$\lfloor {m/ \ell}\rfloor$ vertex-disjoint copies
of $K_{\ell}$. The same can be said about $G^{\prime}-e$, except that $e$
is missing from one copy of $K_{\ell}$. Let $A$ denote its vertex set.

\begin{claim}
There is another copy of $K_{\ell}$ among the ones above, its vertex
set denoted by $B$, such that $G^{\prime}-e$ contains a complete bipartite
$K_{\ell ,\ell}$ spanned by $A$ and $B$. \end{claim}

\begin{proof} In order to prove the claim suppose the contrary: none of the copies of $K_{\ell}$
can be chosen as $B$. If the total number of edges starting from $A$ and ending at
a fixed copy is $\ell^2$ then we found a  $K_{\ell ,\ell}$. Hence, it can be supposed
that this number of edges is at most $\ell^2-1$. The number of edges starting in $A$
and ending in any copy of $K_{\ell}$ is at most
$\left(\lfloor {m/\ell}\rfloor-1\right)(\ell^2-1)$. But there are at most $\ell -1$
vertices not included in the $K_{\ell}$ copies. Therefore, the total number of edges having
exactly one end in $A$ is at most
$$\left(\left\lfloor {m\over \ell}\right\rfloor-1\right)(\ell^2-1)+\ell (\ell -1).$$
Each vertex of $A$ is also connected to the $\ell-1$ other vertices of $A$. These imply  that $A$ has a vertex of degree at most
$${\left(\lfloor {m\over \ell}\rfloor-1\right)(\ell^2-1)\over \ell}+ (\ell -1)+(\ell -1).$$
This can be upper-bounded by
$$m{\ell^2-1\over \ell^2}+ 2(\ell -1).$$
If $m>{2(\ell -1)\over \alpha}$ then this is smaller than the minimum degree condition (\ref{eq31}). This contradiction proves the claim.
\end{proof}

The Claim implies that $A\cup B$ spans a complete graph minus one edge $e$. It is obvious that this
contains two vertex disjoint copies of $K_{\ell}$. Replacing $A$ and $B$ by them, an
almost-$\ell$-decomposition of $G^{\prime}-e$ is obtained,
contradicting our assumption. 
\end{proof}

The vertex sets of the complete graphs $K_{\ell}$ will be called {\it classes}
and will be denoted by upper case letters.

\begin{lem}\label{lem2} Under the conditions of Lemma \ref{lem1}, for each copy of $K_{\ell}$ with
vertex set $C$, there are at least
\begin{equation}\label{eq32}
m\ell \alpha -\ell (\ell -1)
\end{equation}
other copies with vertex sets
$D_1, D_2,\ldots$ such that $C\cup D_i$ spans a complete graph on $2\ell$ vertices.
\end{lem}

\begin{proof} Let $r=r(C)$ be the number of proper $D$s for the given $C$. The number of edges
having one end in $C$ is at most
\begin{equation}\label{eq33}
	r\ell^2+\left( \left\lfloor {m\over \ell}\right\rfloor -1-r \right)(\ell^2- 1)+
\ell (\ell -1)\leq r+{m\over \ell}(\ell^2- 1)+\ell (\ell -1).
\end{equation}
However, (\ref{eq21}) gives a lower bound on this quantity:
\begin{equation}\label{eq34}
	m\ell \left({\ell^2-1\over \ell^2}+\alpha \right).
\end{equation}
Comparing (\ref{eq33}) and (\ref{eq34}) the desired inequality (\ref{eq32}) is obtained.
\end{proof} 

Two classes spanning a $K_{2\ell}$ will be called a {\it compound pair}.

\begin{lem}\label{lem3} Let $G_1$ and $G_2$ be two graphs satisfying the conditions
of Theorem \ref{thm2}. Suppose that $m$ is large enough and let $A$ and $B$ be
the vertex sets of two classes in the
almost-$\ell$-decomposition of  $G_1$. Let $a\in A$ and
$b\in B$ where  $\{ a,b\}$ is not an edge in $G_1\cup G_2$ (i.e. it is neither blue nor red). 

If $G_1$ and $G_2$ do not form an
alternating-$(\ell,2,\ell,2)$-bag then there exist a class $C$ and an element
$c\in C$ such that $A$ and $C$ form a compound pair, furthermore $A-a+c, C-c+a$
and the other classes form an almost-$\ell$-decomposition without an
alternating-$(\ell,2,\ell,2)$-bag, even if the edge $\{ a,b\}$ is added to $G_2$ (i.e. becomes red).
\end{lem}

\begin{proof} By Lemma \ref{lem2} there are at least $m\ell \alpha -\ell (\ell -1)$ classes $C$ forming a double
pair with $A$. Interchanging $a$ and $c$ in them keeps this property,
therefore they cannot form an alternating-$(\ell,2,\ell,2)$-bag.
We will see that the number of classes $C$ not satisfying the other properties
 in Lemma \ref{lem3} is less than $m\ell \alpha -\ell (\ell -1)$.

\begin{claim} The number of $C$s such that $C-c+a$ forms an alternating-$(\ell,2,\ell,2)$-bag
with some class $D$ is less than
\begin{equation}\label{eq35}\ell \Delta_2^2.
\end{equation}
\end{claim}

\begin{proof}
In order to prove the claim, let us first notice that there are no two red edges between
$C$ and $D$ therefore, if there are two red edges between $C-c+a$ and $D$ then one of the
red edges must start at $a$. Let $d\in D$ be the other end of this edge, moreover let
the other red edge be $\{f,g\}$ where $f\in D, g\in C, g\not= c$. Give an upper estimate
on the number of paths $\{ a,d,f,g\}$. Starting from $a$ we have at most $\Delta_2$
choices for $d$. Since both $d$ and $f$ are in $Z$, $f$ can be chosen in $\ell-1<\ell$ ways.
Finally, the last red edge from $f$ to $g$ can be chosen in at most $\Delta_2$ ways. Their
product, $\ell \Delta_2^2$ is a strict upper estimate on the number of these paths, but this must be
also an upper estimate on the number of possible $C$s, proving the claim.
\end{proof}

It is also possible that $C$  is such that there is a class $D$ for every
choice of $c\in C$ which forms an alternating-$(\ell,2,\ell,2)$-bag with $A-a+c$.
Then there is a red edge between $s\in A, s\not= a$ and $t\in D$ and another red edge
between $u\in D$ and $c$. We have a red-blue-red path from $s$ to $c$. The number of these
paths is less than $\ell ^2 \Delta_2^2$. But $C$ is a bad choice here only if there is such
a path to its every element $c\in C$. Therefore the above estimate can be divided by $\ell$.
The number of bad $C$s is in this case is less than $\ell \Delta_2^2$, again.

Finally, we give an upper estimate on the number of possible $C$s creating an
alternating-$(\ell,2,\ell,2)$-bag when $\{a,b\}$ turns red. This edge is connecting $B$ and $C-c+a$,
therefore there must be another red edge between $B$ and $C$. The number of choices is less than
\begin{equation}\label{eq36}\ell \Delta_2. 
\end{equation}	

By (\ref{eq35}) and (\ref{eq36}) the total number of bad $C$s is less than
$$ 2\ell \Delta_2^2+\ell \Delta_2< 3\ell \Delta_2^2.$$
There is a good $C$ if (\ref{eq34}) is not smaller:
$$m\ell \alpha -\ell (\ell -1)\geq \ell \Delta_2^2.$$
But this is a consequence of (\ref{eq22}). \end{proof}

\begin{proof}[Proof of Theorem \ref{thm2}] Suppose that the conditions of the theorem hold for the graphs
$G_1$ and $G_2$. If $G_2$ is replaced by the empty graph then Lemma \ref{lem1} ensures the existence
of an almost-$\ell$-decomposition of $G_2$ without an  alternating-$(\ell,2,\ell,2)$-bag.
 Add edges of $G_2$ one by one until we cannot find such an almost-$\ell$-decomposition of $G_1$.
 That is, adding the red edge $\{a,b\} ( a\in A, b\in B)$ there is no proper solution, but without it
 there is one. Take this latter almost-$\ell$-decomposition of $G_1$ and apply Lemma \ref{lem3}  for it.
 An  almost-$\ell$-decomposition is obtained without $A_1\cap B_1$, although
 $\{a,b\}$ is included. This contradiction proves the statement. \end{proof} 

\begin{proof}[Proof of Theorem \ref{thm1}] The result of Theorem \ref{thm2} is used. Let the vertex set be
 $V={[n]\choose k}$, the family of all $k$-element subsets of $[n]$. Two vertices are adjacent in $G_1$
 if the corresponding $k$-element subsets are disjoint. Then a $(k,\ell )$-papartition
 corresponds to a copy of $K_{\ell}$ in $G_1$. Two vertices in $G_2$ are adjacent if
the intersection of the corresponding $k$-element subsets has a size more than ${k/2}$.
Two such papartitions are too close if and only if the corresponding copies of $K_{\ell}$ form
an  alternating-$(\ell,2,\ell,2)$-bag.
Therefore we only have to check the validity of the inequalities (\ref{eq21}) and (\ref{eq22}) for this graph.

$G_1$ is regular of degree ${n-k\choose k}$, therefore if we choose $\alpha ={1\over 2\ell^2}$ then
(\ref{eq21})  becomes
$${n\choose k}\left(1-{1\over 2\ell^2}\right)\leq {n-k\choose k}.$$
The ratio of the two binomial coefficients tends to 1 with $n$, hence the inequality holds for large $n$s.

$G_2$ is also regular, its degree is
$$\sum_{1\leq i<{k\over 2}}{k\choose i}{n-k\choose i}\leq O(n^{{k-1\over 2}}).$$
But the right hand side of (\ref{eq22}) in this case is asymptotically a constant times $n^{k\over 2}$,
 proving the inequality and the theorem. \end{proof} 

\section{Remarks, problems}
The story has started with the following theorem of Dirac.
\begin{thm} {\rm \cite{D}} If the minimum degree in a graph of $m$ vertices is at least ${m\over 2}$ then the graph contains a Hamiltonian cycle (going through every vertex exactly once).
\end{thm}
Actually, our Lemma \ref{lem1} is a very weak form of Dirac's theorem if $\ell=2$. The condition is more than ${3\over 4}m$
rather than ${m\over 2}$ and we obtained only every second edges of the Hamiltonian cycle. But there is a Dirac type theorem
for larger $\ell$. The $p$th power $H^p$ of a Hamiltonian cycle $H$  is obtained by joining every pair of vertices not farther than $p$ along the cycle. Seymour \cite{S} conjectured, Koml\'os, G. S\'ark\"ozy and Szemer\'edi \cite{KSSz} proved that if the minimum degree in a graph is at least $m{p-1\over p}$ and $m$ is large enough, then the graph contains $H^p$ as a subgraph. Our Lemma \ref{lem1} is a weakening of this theorem
($p=\ell-1$). We proved here this weaker version because its proof is much easier than the really deep proof in \cite{KSSz}. Also, it does not need such a very large value of $m$.

This comparison opens a new problem: find an extension of Theorem \ref{thm2} in this direction.
\begin{prob} Under what degree conditions can we say that $G_1$ contains an $H^{\ell -1}$ such that it does not form an 
alternating-$(\ell,2,\ell,2)$-bag with two red edges?
\end{prob}

\begin{prob} Improve the constants in our theorems, decreasing the thresholds in $m$ and $n$.  
\end{prob}

We are sure that our results do not reach the boundaries of the method. Define the {\it size} of a $(k,\ell)$-papartition as $k\ell $. 
\begin{prob} For given $n$ maximize the size of the papartitions satisfying the condition in Theorem \ref{thm2}.
Can it be constant times $n$?
\end{prob} 

Can it be $n$? Even if the answer is yes, the present method is not strong enough to prove it. If this is too hard to prove, there might be other similar
conditions instead of our concept ``too close''.

\begin{prob} (Baranyai theorem with an additional condition.) Suppose $k|n$. Find a non-trivial relation $R$ between two partitions of an $n$-element set into $k$-element subsets in such a way that the family of all $k$-element subsets can be decomposed into such partitions in such a way that no two of them are in relation with rspect to $R$.
\end{prob}

There are many open problems related to the Baranyai theorem. Let us popularize an old and difficult one.
First a {\it wreath} is defined as follows.  Take a cyclic permutation of the elements of $[n]$ and consider only intervals along the cyclic permutation. Choose one interval, then take the interval starting immediately after the end of the first interval. The third one starts after the end of the second interval, and so on.
If $k|n$ then this ends after choosing ${n/ k}$ intervals and in this case  the wreath consists of these ${n/ k}$ intervals.
In general we do not stop after the first round, only when our interval fits to the initial interval. That is if ${\rm lcm}(n,k)$ is the least common multiple then we stop
after ${{\rm lcm}(n,k)/ k}$ intervals and go around ${{\rm lcm}(n,k)/n}$ times.

\begin{conj}[The wreath conjecture.] { The family of all $k$-element subsets of $[n]$ can be decomposed into disjoint wreaths.} 
\end{conj}

This was jointly conjectured by Zsolt Baranyai and  the first author. Baranyai tragically passed away in 1976, the conjecture appeared in print only in 1991 \cite{ohK}.
Later, independently, Bailey and Stevens posed related conjectures, see  paper \cite{JGO} for the  relevant references.

Finally let us mention that paper \cite{ohK2} contains some related earlier results and open problems.

\end{document}